\documentclass{amsart}
\usepackage{latexsym}
\usepackage{amssymb}
\usepackage{amsmath}
\usepackage{amsthm}
\usepackage{setspace}
\usepackage{pict2e}
\usepackage{graphicx}
\usepackage{subfigure} 
\usepackage{algorithm,algorithmic}
\usepackage{tabularx,ragged2e}
\usepackage{lineno}
\usepackage[bb=boondox]{mathalfa}

\newtheorem{thm}{Theorem}[section]
\newtheorem*{equivlemma}{Equivalence Lemma}
\newtheorem*{symconj}{Symmetry Conjecture}
\newtheorem{lemma}[thm]{Lemma}

\newtheorem{cor}[thm]{Corollary}
\newtheorem{prop}[thm]{Proposition}

\newtheorem{defn}[thm]{Definition}
\theoremstyle{definition}

\newtheorem{m*N}[thm]{}

\newcommand\ts{\hspace{2em}}

\def\addcontentsline#1#2#3{\relax}

\long\outer\def\demo#1. #2\par{\medbreak\noindent {\bf#1.\enspace}
        {\rm#2}\par\ifdim\lastskip<\medskipamount\removelastskip
        \penalty55\medskip\fi}
\def\bdemo#1. #2\par{\medbreak\noindent {\bf#1.\enspace}{\rm#2}\par}
\def\edemo{\ifdim\lastskip<\medskipamount\removelastskip\penalty55\medskip\
fi}

\def\Me{\mathbb{e}}
\def\Mi{\mathbb{i}}

\def\N{{\mathbb{N}}}

\def\calP{\mathcal{P}}
\newcommand\TI[1]{\textit{\textbf{#1}}}

\begin{document}
\title{Littlewood Polynomials, Spectral-Null Codes, and Equipowerful Partitions}
\author{ Joe Buhler, Shahar Golan, Rob Pratt, and Stan Wagon }

\begin{abstract}
Let $[n]$ denote $\{0,1, ... , n-1\}$.  
A polynomial $f(x) = \sum a_i x^i$ is a Littlewood polynomial (LP) of length $n$ 
if the $a_i$ are $\pm 1$ for $i \in [n]$, and $a_i = 0$ for $i \ge n$.
Such an LP is said to have order $m$ if it is divisible by $(x-1)^m$.
The problem of finding the set $L_m$ of lengths of LPs of order $m$ 
is equivalent to finding the lengths of spectral-null codes of order $m$,  
and to finding $n$ such that $[n]$ admits a partition into two subsets whose first $m$ moments are equal.  
Extending the techniques and results of Boyd and others, we completely determine 
$L_7$ and $L_8$ and prove that 192 is the smallest element of $L_9$.  
Our primary tools are the use of carefully targeted searches using integer 
linear programming (both to find LPs and to disprove their existence for 
specific $n$ and $m$), and an unexpected new concept (that arose out of observed
symmetry properties of LPs) that we call ``regenerative pairs,'' 
which produce infinite arithmetic progressions in $L_m$.  
We prove that for $m \le$ 8, whenever there is an LP of length $n$ 
and order $m$, there is one of length $n$ and order $m$ that is symmetric 
(resp.~antisymmetric) if m is even (resp.~odd).
\end{abstract}
\maketitle

\setcounter{section}{0}
\setcounter{equation}{0}
\setcounter{thm}{0}

\let\thefootnote\relax\footnotetext{ 2000 Primary 11B83, 12D10; Secondary 94B05, 11Y99.
}

\let\thefootnote\relax\footnotetext{ 
Joe Buhler: Department of Mathematics, University of Minnesota, Minneapolis, MN 55455; Email: jbuhler@umn.edu.
Shahar Golan: Department of Computer Science, Jerusalem College of Technology, Jerusalem, Israel;
E-mail: sgolan@jct.ac.il.
Rob Pratt: Cary, NC 27513; E-mail: rob.pratt@sas.com.
Stan Wagon: Macalester College, St. Paul, MN 55105; E-Mail: wagon@macalester.edu.
}

\section{Introduction}\label{intro}

The partition of $\{0,1,2,3,4,5,6,7\}$ into the two disjoint sets 
$A=\{0,3,5,6\}$ and $B=\{1,2,4,7\}$ is especially well-balanced ({``}equipowerful{''}) 
in that the first three moments of $A$ are equal to the corresponding moments of $B$; 
i.e., $\sum_{a\in A}a^j=\sum_{b\in B}b^j=\frac{1}{2}\sum_{\, i=0}^{7}i^j$ for $j=0,1,2$. 
Figure 1 illustrates the three identities geometrically. We say that the
bipartition is 3-equipowerful (or has order 3) and length 8.

\begin{figure}[!ht] \label{figure1}
  \centering
    \includegraphics[clip=true,trim = 100 0 0 0,scale=1.00]{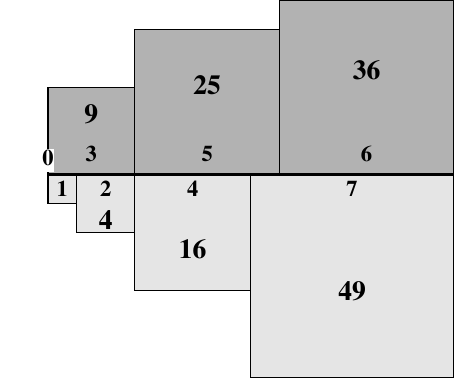}
  \caption{A 3-equipowerful bipartition of length 8; the three power-sums are 4, 14, 70.}
\end{figure}

A bipartition of this kind can be encoded algebraically as a generating function 
in two natural ways: either as a Littlewood polynomial (all coefficients are $\pm 1$; 
we abbreviate this to LP) via $f(x)=\sum_{a\in A}x^a-\sum_{b\in B}x^b$ 
or as the trigonometric polynomial $g(t)=f(\Me^{\Mi t})=\sum_{a\in A}\Me^{\Mi at}-\sum_{b\in B}\Me^{\Mi bt}$, 
all of whose coefficients are $\pm 1$. The fact that the preceding bipartition is 
3-equipowerful is equivalent to $f(x)$ vanishing to order 3 at $x=1$, and also to $g(t)$ vanishing
in $t$ to order 3 at $t=0$.

These trigonometric polynomials correspond to ``spectral-null codes'' which arise in 
signal processing in several important ways, e.g., in encoding digital information 
on media (such as a DVD) where low frequencies need to be suppressed.  Applications
for small~$m$ use efficient algorithms for encoding arbitrary bit strings into (the somewhat
longer) strings in some spectal-null code of order~$m$.  Using codes for larger~$m$ will
require reasonably efficient
encoding/decoding algorithms, and a better understanding of what lengths are possible
for codes of some specific order~$m$.

This latter question, or equivalent versions in other guises, is the central focus of this paper: 
given a positive integer $m$, what are the possible lengths of (equivalently) an equipowerful 
bipartition, a Littlewood polynomial, or a spectral-null code, of order $m$? 
This set will be written
\[
 L_m=\{ n:\text{there is an LP of length } n \text{ and order } m\}.
\]
This has been investigated in the Littlewood polynomial context by, among others, 
Boyd \cite{boyd1, boyd2}, Berend and Golan \cite{golan}, and Freiman and Litsyn \cite{asymptotically}.

The primary tools that we use to significantly extend known results are (1) carefully 
targeted searches that use integer linear programming (ILP), and (2) an unexpected 
concept that we call {``}regenerative pairs,{''} which yields efficient proofs of 
the existence of infinite families of order-$m$ LPs and suggests additional structure 
for LPs of high order.

Section 2 gives the required definitions, $\S $3 states our results, $\S $4 gives the 
background theorems that are needed, $\S $5 has proofs of the negative results, and 
$\S $6 introduces the concept of regenerative pairs, which are used to find infinite 
families of LPs.
The final section states some conjectures that emerge from the data. 
Our work shows how one can use experiments, aided by sophisticated computations 
(ILP) to generate hunches based on patterns, which can be used 
to refine the experiments and also as a guide to proofs. 
One danger of course is that the patterns are not as predictive as one might hope and
we enumerate a series of failed conjectures, including one with a spectacular counterexample
that has order 52 and length~$2^{51}$.

\section{Notation and Definitions}\label{Defs}

We use $[n]$ to denote $\{0,1,\, \dots ,n-1\}$. 
A decomposition $[n]=A\cup B$ into two disjoint subsets is called $m$\textit{-equipowerful}, 
or \textit{of order }$m$, if the first $m$ moments (power-sums) of the sets are equal, 
i.e., $\sum_{a\in A}a^j=\sum_{b\in B}b^j$ for $j=0,1,\dots ,m-1$. 
We write $A\overset{m}{=}B$ when the first $m$ moments of $A$ and $B$ are equal; 
$m$ is called the \textit{order} of the bipartition and $A$ and $B$ are called \textit{witnesses}. 
We always assume that $m$ is positive, so that the two subsets have the same size. 
The equations $A\overset{m}{=}B$ are an example of a {``}multigrade identity{''} and also 
an instance of the general Prouhet--Tarry--Escott problem \cite{tarryescott, prouhet}.

A \textit{Littlewood polynomial} (LP) is a polynomial whose exponents are $[n]$ for
some $n$ and whose coefficients are all $\pm 1$.  The degree of a polynomial $f$ is 
denoted $\deg (f)$; the \textit{length} of a polynomial $f$, denoted $\ell (f)$, 
is the number of its coefficients, i.e., $\ell(f) = \deg (f)+1$. 
An LP $f(x)$ has order $m$ if $(1-x)^m$ divides $f(x)$.  
Note that order~$m$ merely means divisible by $(x-1)^m$, and we will have occasion
to refer to the \textit{exact order} of an LP $f$, meaning the maximum order, i.e., $m$ is
the exact order if not only $(x-1)^m$ divides $f$, but also $(x-1)^{m+1}$ does not divide~$f$.

As will be shown below, equipowerful partitions of $[n]$ of order $m$ and LPs 
of length~$n$ and order $m$ are equivalent ideas.
We will always assume that $m$ is positive, which implies that $n$ is even.
The ideas extend immediately to bipartitions of any interval $I$ of $n$ consecutive
integers (or to LPs with nonzero coefficients in~$I$).
Multiplying an LP by a power of $x$ does not change the power of $x-1$ that divides it, so
it suffices to consider polynomials with $I = [n]$ and,
by the aforementioned equivalence (or a direct proof), it suffices to consider bipartitions of~$[n]$.

We use $\calP (n,m)$ for the set of LPs of length $n$ and order $m$, and $L_m$ for the 
set of lengths of LPs of order $m$: $L_m=\{n:\calP (n,m)\neq \varnothing\}$. 
The initial goal of this paper is to discover as much as possible about $L_m$, for small $m$.

A dual concept is also useful: let $m^*(n)$ be the largest $m$ so that $n\in L_m$. 
For example, $m^*(8)=3$ because the example of $\S $1 gives 
$1-x-x^2+x^3-x^4+x^5+x^6-x^7=(1-x)^3 (1+x)^2 (1+x^2)$, which is in $\calP(8,3)$ and 
a short search shows that $\calP(8,4)$ is empty.

The \textit{join} of LPs $f(x)$ and $g(x)$ is $(f\lor g)(x)=f(x)+x^{\ell (f)}g(x)$.
The \textit{expanded product} of LPs $f$ and $g$, denoted $f\#g$, is $(f\# g)(x)=f(x)g(x^{\ell (f)})$, 
which is an LP; this product is associative but not commutative. 

Symmetry plays a large role in our investigations, so we introduce several concepts 
related to symmetric and antisymmetric LPs. The \textit{reversal} of an LP $f(x)$ 
is $f^*(x)=x^{\deg (f)}f(1/x)$ and $f$ is \textit{symmetric} (or $+$\textit{symmetric}) 
if $f^*=f$ and \textit{antisymmetric} (or $-$\textit{symmetric}) if $f^*=-f$. 
If $s=1$ or $s = -1$, then the $s$\textit{-symmetrization} (or \textit{symmetrization} if $s=1$ 
and \textit{antisymmetrization} if $s=-1$) of an LP $f$ is $S(f)=f\lor (sf^*)$. 

Finally, we introduce a special polynomial that is a seminal example in our story.

\begin{defn} The \textit{Thue--Morse polynomial}, which we denote $\tau_m$, is 
\[
\tau_m(x) = (1-x)(1-x^2)\dots \left(1-x^{2^{m-1}}\right).
\]
\end{defn}

The polynomial $\tau_m$ has length $2^m$ and is $(-1)^m\,$symmetric. 
These polynomials played an important role in the early work because $\tau_m$ has order $m$, 
which means $\tau_m\in \calP (2^m,m)$, $2^m\in L_m$, and $m^*(2^m)\geq m$.

\section{New Results}\label{new}

Work of Boyd \cite{boyd1,boyd2}  and Berend and Golan \cite{littlewood} gave exact descriptions of $L_m$ for $m\leq 6$, 
and also exact values of $m^*(n)$ for $n\leq 167$ and $n=256$; these are included in the
tables below.

Their results for $L_m$ are extended here to exact descriptions for $m = 7$ and $m = 8$, as well
a determination of the elements of $L_9$ that are smaller than 272.  The extensive and
carefully structured computations required to do this will be described in $\S $5 and $\S $6.
The next theorem states the new results precisely, and the full story is given in the
subsequent table.  We let $\N$ denote the set of nonnegative
integers.

\begin{thm}\label{3.1}
$L_7=\{96, 112, 128, 144, 160, 176\} \cup  (192+8\, \N)$ 
and $L_8=\{144\} \cup  (192+16\, \N)$.
Also, the only integers in $L_9$ that are less than 272 are 192 and 240.
\end{thm}

\begin{table}[ht]\label{lmTab1}
$
\begin{array}{|l|l|} \hline
 m & L_m \\ \hline
 1 & 2+2\, \N \\ \hline
 2 & 4+4\, \N \\ \hline
 3 & 8+4\, \N \\ \hline
 4 & 16+8\, \N \\ \hline
 5 & 32+8\, \N \\ \hline
 6 & \{48\} \;  \cup \;   64+8\, \N \\ \hline
 7 & \{96, 112, 128, 144, 160, 176\} \;  \cup \;  192+8\, \N \\ \hline
 8 & \{144\} \; \cup \;  192+16\, \N \\ \hline 
\end{array}
$
\vspace*{0.1in}
\caption{$L_m$, for $m \le 8$.}
\end{table}

Note that in Table 1 the limiting differences of the sets $L_m$ are 2, 4, 4, 8, 8, 8, 8, 16. 
In fact, the Divisibility condition in Theorem~\ref{basic} below says if $\calP (n,m)\neq \varnothing$ 
then $n$ is divisible by the smallest power of 2 that is strictly bigger than $m$. 
For example, any number in $L_2$ or $L_3$ has to be divisible by 4; and $L_4$ 
through $L_7$ are contained in $8\,\N$. 
Remarkably, this divisibility criterion is not only necessary for membership in $L_m$, 
but it is (almost) sufficient. 
Freiman and Litsyn \cite[Thm.~1]{golan} proved that, roughly speaking, the chance that a 
random LP of length $n$ is divisible by $(x-1)^m$ is about $n^{-m^2/2}$; more precisely,
$\left| \calP (n,m)\right| =c\, (1+o(1))\,  2^n/n^{-m^2/2}$, where the 
constant $c$ depends only on $m$ and the $o(1)$ term goes to zero as $n$ goes to infinity. 
This means that the divisibility condition is true asymptotically in the sense that $L_m$ contains all 
sufficiently large $n$ that are divisible by the appropriate power of 2. 
Unfortunately, the implicit constants in the $o(1)$ term are unknown, so this doesn't 
help us find $L_m$ explicitly for small $m$. 
Consider $L_7$ and $L_8$. The Freiman--Litsyn result implies that $L_7$ contains 
all sufficiently large multiples of 8, and $L_8$ contains all sufficiently large multiples of~16.
Theorem~\ref{3.1} asserts that $L_7$ contains all multiples of 8 that are 192 or greater 
and $L_8$ contains all multiples of 16 that are 192 or greater.
As a side note,
we observe that the Freiman--Litsyn estimate can be used to coarsely estimate the smallest $n$ for which $\calP(n,m)$
is nonempty; we find that for a given~$m$ the smallest element of $L_m$ is, asymptotically, roughly 
equal to $4m^2$.

Our work reveals further structure in equipowerful sets. Anyone working in the area is drawn to the ubiquitous symmetry properties of the
witnessing sets and polynomials. 
That is, whenever $n\in L_m$, it appears that there is a witnessing set that is $(-1)^m$ symmetric. 
This is easy to prove when $m=2$: just apply symmetry to extend all the initial segments 
of the infinite sequence $X=(0,1,2,3,\dots )$, getting $\{0,3\}$, 
$\{0,1,6,7\}$, $\{0,1,2,9,10,11\}$, and so on. Yet this had not been proved even for $m=3$. 
However, we now know that $(-1)^m\,$symmetric
witnesses exist for all entries in our table of~$L_m$.

\begin{thm}\label{3.2}
In all the cases of Table 1, there is a $(-1)^m\,$symmetric LP in $\calP (n,m)$.
\end{thm}

Recall that $m^*(n) = \max\{m : \calP(n,m) \ne \varnothing \}$, i.e., $m = m^*(n)$ is the largest
power of $x-1$ that divides some LP of length~$n$.
Earlier work and Theorem~\ref{3.1} determine $m = m^*(n)$ for $m \le 7$.
The $m^*(256)$ value in the $m^*$ table was known \cite{littlewood}, and our calculations 
found the following values.

\begin{thm}\label{3.3} 
$m^*(192)=9$ and $m^*(208)=m^*(224)=8$.
\end{thm}

Note that $m = 240$ is the smallest value for which $m^*(n)$ is not known exactly.
The remarkable $2^{51}$ result is due to Richard Stong (Theorem~\ref{stong}), as will be
described in detail later.

\begin{table}[ht]\label{lmTab2}
$
\begin{array}{|c|c|c|c|c|c|c|c|c|c|c|c|c|c|c|c|c|c|c|} \hline
 n & 8 & 16 & 24 & 32 & 40 & 48 & 56 & 64 & 72 & 80 & 88 & 96 & 104 & 112 & 120 & 128 & 136 & 144 \\ \hline
m^*(n)  & 3 & 4 & 4 & 5 & 5 & 6 & 5 & 6 & 6 & 6 & 6 & 7 & 6 & 7 & 6 & 7 & 6 & 8\\ \hline
\end{array}
$
\hfill \\ \vspace*{0.1in}
$
\begin{array}{|c|c|c|c|c|c|c|c|c|c|c|c|c|c|c|c|} \hline 
 n & 152 & 160 & 168 & 176 & 184 & 192 & 200 & 208 & 216 & 224 & 232 & 240 & 248 & 256 & \vphantom{\rule{0pt}{14pt}}2^{51} \\ \hline
m^*(n)  & 6 & 7 & 6 & 7 & 6 & 9 & 7 & 8 & 7 & 8 & 7 & \text{$\geq $9} & 7 & 8 & \vphantom{\rule{0pt}{12pt}} \text{$\geq $52} \\ \hline
\end{array}
$ \\ \vspace*{0.05in}
\caption{$m^*(n)$.}
\end{table}

\section{Basic Theorems}\label{thms}

We start with several classic results about equipowerful sets. 
The first gives the important equivalence between LPs, equipowerful sets, and spectral-null codes.

\begin{equivlemma}
For a bipartition $A,B$ of $[n]$ and an integer $m\geq 1$, the following are equivalent.
\begin{enumerate}
\item $A\overset{m}{=}B$.
\item $t^m$ divides the power series of $g(t)=\sum_{a\in A}\Me^{\Mi at}-\sum_{b\in B}\Me^{\Mi bt}$.
\item $(x-1)^m$ divides the polynomial $f(x)=\sum_{a\in A}x^a-\sum_{b\in B}x^b$.
\end{enumerate}
\end{equivlemma}

\begin{proof} { }$1\Leftrightarrow 2$. { }Let $a_i$ be the $\pm $1 sequence of length 
$n$ that corresponds to $A,B$. 
Then (2) holds if and only if the first $m$ derivatives $g^{(j)}(0)$, $0\leq j\leq m-1$, all equal 0. 
But $g^{(j)}(0)=\sum_{i=0}^{n-1} a_ii^{j}$, 
so the vanishing of the derivatives is equivalent to $A\overset{m}{=}B$.

$2\Leftrightarrow 3$. { }An easy induction argument, using the chain and product rules, 
shows that there are integers $s_{j,k}$ so that
\[
 g^{(k)}(t)=f^{(k)}(x)(i x)^k+i^k\sum_{j=1}^{k-1}s_{j,k}\text{ }f^{(j)}(x) \, x^j,
\]
where $x=\Me^{\Mi t}$.  If all of the $j$th derivatives $g^{(j)}(0)$ and $f^{(j)}(1)$ vanish for $j\leq k-1$, 
then it follows that $g^{(k)}(0)=0$ and $f^{(k)}(1)=0$ are equivalent, and therefore $g$ has order $m$ 
at 0 if and only if $f$ has order $m$ at 1. 
\end{proof}

The next theorem collects various known results. The Multiplication result is in \cite[Prop.~2.4]{littlewood}.

\begin{thm}\label{basic} \phantom{fakelinesothatitcanbe broken...} \hfill \vspace*{0.03in} \\ \vspace*{0.05in}
\hspace*{0.1 in} {\bf Thue--Morse.}  $\tau_m \in \calP(m,2^m)$, and therefore $2^m\in L_m$. \hfill \\ \vspace*{0.05in}
\hspace*{0.1in} {\bf Addition.}  If $f\in \calP(n_1,m)$ and $g\in \calP(n_2,m)$ 
then $f \lor g \in \calP(n_1+n_2,m)$ \hfill \\ \vspace*{0.05in}
\hspace*{0.1in} {\bf Doubling.}  If $f \in \calP(n,m)$ then $f \lor -f \in \calP(2n,m+1)$. \hfill \\ \vspace*{0.05in}
\hspace*{0.1in} {\bf Multiplication.} If $f\in \calP (n,m)$ and 
$g\in \calP (n',m')$, then $f\#g\in \calP (n\, n',m+m')$. \hfill \\ \vspace*{0.05in}
\hspace*{0.1in} {\bf Symmetry.} If $f$ is a $(-1)^m$ symmetric LP that is divisible by
$(x-1)^{m-1}$ then $f$ is divisible by $(x-1)^m$.  \hfill \\ \vspace*{0.05in}
\hspace*{0.1in} {\bf Divisibility.} Let $n=2^k\, u$ where $u$ is odd.  If $\calP(n,m)$ is nonempty then $2^k > m$. \\
\end{thm}

\begin{proof} \textit{Thue--Morse.} 
Since $\tau_{m+1}=\tau_m\lor -\tau_m$ it follows that $\tau_m$ is an LP, and
length $2^m$ and order $m$, by induction.  Note also that $\tau_{m+1} =\tau_m\# (1-x)=(1-x) \#\tau_m$.

\noindent \textit{Addition.} The join $f \lor g$ is obviously an LP, has length $n_1+n_2$ and is divisible 
by $(x-1)^m$, as asserted.  Note that this implies that $L_m$ is closed under addition.

\noindent \textit{Doubling.} If $f$ is in $\calP (n,m)$ then 
$g=(1-x^n)f(x) = f\lor (-f)$ has length $2n$ and is divisible by $(1-x)^{m+1}$ since
$1-x^n$ introduces an extra $1-x$ factor.  In other words, if $n \in L_m$ then $2n \in L_{m+1}$.

\noindent \textit{Multiplication.} The polynomial $h(x)=f(x)g(x^n)$ 
consists of $n'$ blocks of length $n$, each the result of multiplying the
coefficients of $f$ by a single coefficient of~$g$.
In particular $h$ is an LP. 
The polynomial $h$ is divisible by $(x-1)^m(x^n-1)^{m'}$ and therefore divisible by $(x-1)^{m+m'}$.
Note that Doubling is the special case $g(x) = 1-x$ of Multiplication.

\noindent \textit{Symmetry.} 
We are given that $f^*(x) = (-1)^m f(x)$, and that there is a polynomial $g$ such that $f(x) = (x-1)^{m-1} g(x)$.
Therefore, using the fact that reversal is multiplicative, we get
\[
(x-1)^{m-1}g(x) =f(x) = (-1)^m f^*(x) = (-1)^m(-1)^{m-1}(1-x)^{m-1}g^*(x) = -(1-x)^{m-1}g^*(x).
\]
Dividing by $(1-x)^{m-1}$ gives $g(x) = - g^*(x)$, which implies that $g(1) = g^*(1) = 0$ so that
$g(x)$ is divisible by $x-1$, from which the result follows.
Note that this implies that if $f$ has exact order $m$, and is symmetric or antisymmetric, then $f$
is $(-1)^m$ symmetric.

\noindent \textit{Divisibility.} The key point is that the number of factors of $1-x$ in $f(x)$ modulo 2 
(i.e., working in the ring of polynomials over the 2-element field $\{0,1\}$) is at least as large as the number 
of such factors when $f$ is thought of as a polynomial with integer coefficients.

Suppose that $f\in \calP (n,m)$. Then $f(x)\equiv 1+x+x^2 +\dots +x^{n-1} \pmod 2 $
and $(1-x)f(x)\equiv 1-x^n$, where $\equiv $ will denote coefficient-wise congruence modulo~2. 
Note that $1-x^2\equiv (1-x)^2$. Iterating this shows that $1-x^t\equiv (1-x)^t$, where $t=2^k$. 
Introducing the shorthand $y=x^t$ gives
\[
(1-x)f(x)\equiv 1-x^n=1-x^{tu}=(1-y)\left(1+y+\dots +y^{u-1}\right)\equiv (1-x)^tg(x),
\]
where $g(1)=u \equiv 1$ so that $g(x)$ is not divisible by $1-x$. 
Counting factors of $1-x$ gives $1+m\leq 2^k$, as claimed.
\end{proof}

\vspace*{0.1in}

\noindent {\bf Remarks}: 
\begin{enumerate}
\item 
Doubling reverses symmetry in the sense that it turns a symmetric
polynomial into an antisymmetric one, and vice versa.
Addition does not preserve either symmetry.

\item Thue--Morse, Addition, and Doubling imply that 
\[
m^*(2^m)\geq m, \quad 
m^*(n_1+n_2)\geq \min \left[m^*(n_1),m^*(n_2)\right], \; \textrm{and} \;
m^*(2n)\geq m^*(n)+1.
\]

\item The sets arising from the $\tau_k$ are initial segments of the infinite 
Thue--Morse sequence $0,3,5,6,9,10,12,15,17,18,20,23,24,27,29,30,\dots$.
This sequence is exactly the set of so-called \textit{evil numbers}: integers with 
an even number of 1s in their binary expansion. 
Note also that this sequence has one entry in common with each pair $\{0,1\},\{2,3\},\{4,5\},\dots$. 
Our heuristic searches often tried to prioritize this feature (or similar Thue--Morse-like properties)
to aim for earlier success.

\item Divisibility says that if $\calP (n,m)\neq \varnothing$, then $n$ has 
to be divisible by the smallest power of 2 that is strictly bigger than $m$. 
As mentioned earlier, the main result of \cite{asymptotically} implies that
this necessary condition is also sufficient for large enough~$n$.
\end{enumerate}

Let's look at at $L_m$ for small $m$. The set $L_1$ consists of all $n$ such that there 
is some length-$n$ LP $f$ that is divisible by $x-1$; this is the same as saying that $f(1)=0$. 
Such an LP exists if and only if $n$ is even, so $L_1$ is the set of positive even integers,
which we write as $2+2\N$, where $\N$ denotes the set of all nonnegative integers.

By Divisibility, $L_2\subseteq 4+4\N$, and  $\tau_2=1-x-x^2+x^3\in \calP (4,2)$ so $4\in L_2$.
Repeatedly joining $\tau_2$, and using Addition, shows that $L_2 = 4+4\N$.

Although $L_3\subseteq 4+4\N$ by Divisibility, an easy search shows 
that no LP of length 4 is divisible by $(x-1)^3$. 
Because $\tau_3\in \calP (4,3)$ we know that $8+8\N\subseteq L_3\subseteq 8+4\N$. 
A computer (or even hand) search finds that
\[
 f=1-x+x^2-x^3-x^4-x^5+x^6+x^7+x^8-x^9+x^{10}-x^{11}\equiv 0 \bmod (x-1)^3.
\]
By the Equivalence Lemma this is equivalent to the fact that $\{0,2,6,7,8,10\}$ is a 
3-equipowerful subset of $[12]$. By repeatedly joining $\tau_3$ onto $\tau_3$ or $f$ 
(i.e., using Addition) it follows that $L_3=8+4\N$. 
This characterization is due to Boyd \cite{boyd1}.

Similar but more elaborate calculations (done in \cite{littlewood}) show that $L_4=16+8\N$ 
and $L_5=32+8\N$. 
The natural guess that $L_6=64+8\N$ turns out to be wrong, as discovered by 
Boyd \cite{boyd1} because there is a (unique up to sign) polynomial in $\calP (48,6)$ 
(i.e., 48 is 6-equipowerful); he also found \cite{boyd2} that 30, 40, and 56 are not 
in the set, so that $L_6=\{48\}\cup (64+8\N)$. 
Theorem 3.1 specifies $L_m$ exactly for $m=7$ and $m= 8$.

A key to understanding the structure of the LPs in $\calP (n,m)$ is generating 
various constraints that the polynomials and the corresponding equipowerful sets must satisfy. 
The Divisibility condition in Theorem 4.1 is an example of such a constraint. 
But there are many more and any computer search for equipowerful sets beyond modest values 
of $n$ requires the use of a wide variety of constraints.

Here we use the following notations. For a set $X$ and prime $p$, $C_{p,j}(X)$ denotes 
the number of elements of $X$ that are congruent to $j \bmod p$, namely 
$C_{p,j}(X)=\left| X\cap (p \mathbb{Z}+j)\right|$. 
For an $m$-equipowerful bipartition of $[n]$ into $A,B$, put $d_{p,j}=C_{p,j}(A)-C_{p,j}(B)$. 
And we use $\TI{d}_p$ for the vector of $d_{p,j}$ values, $j=0,1,\dots ,p-1$.

The preceding definition formalizes the idea of the discrepancy between the sets in a witnessing bipartition. 
If $\TI{d}_2=(a,b)$, then the even numbers in $A$ exceed those 
in $B$ by $a$, while the odds in $A$ exceed the odds in $B$ by $b$. 
So if $\TI{d}_2=(0,0)$, then the even and odd counts are the 
same for $A$ and $B$; in this case each set would have $n/4$ evens and $n/4$ odds. 
This uniformity of parity happens for the sets arising from the Thue--Morse polynomials.

The next lemma gives several basic constraints regarding the value of $d_{p,j}$.

\begin{lemma}\label{4.2}
Suppose $d_{p,j}$ are defined from an $m$-equipowerful bipartition $A\cup B$ of $[n]$. Then:
\begin{enumerate}
\item { }$\left| d_{p,j}\right| \leq C_{p,j}([n])\leq \left\lceil \frac{n}{p}\right\rceil$.
\item { }$\sum_{j=0}^{p-1}d_{p,j}=0$.
\item { }For each $j$, $d_{p,j}-C_{p,j}([n])$ is even.
\end{enumerate}
\end{lemma}

\begin{proof} (1) holds because the extreme case has all the $j \bmod p$ numbers in one of the sets.
(2) reduces to $\left| A\right| =\left| B\right|$. 
For (3), $d_{p,j}=C_{p,j}(A)-C_{p,j}(B)\equiv C_{p,j}(A)+C_{p,j}(B)=C_{p,j}([n]) \pmod 2 $. 
\end{proof}

The following propositions, with the same assumptions about $m,n,A,B$, provide constraints on~$d_{p,j}$ 
that depend on $m$ and the prime $p$. 
The results are immediate consequences of items 3.3 and 3.4 in \cite{littlewood}.
The proof method in \cite{littlewood} builds upon ideas introduced by Boyd that involve substituting
roots of unity for $x$ and then using facts about the cyclotomic fields generated by those roots
of unity.

\begin{prop}  Suppose $m=s(p-1)+r$, with $0\leq r\leq p-2$.  Then:
\begin{enumerate}
\item { }If $r=0$, then $d_{p,0},d_{p,1},\dots ,d_{p,p-1}$ are all congruent modulo $p^s$. \vspace*{0.1in}

\item { }If $r\geq 1$, then each $d_{p,j}$ is divisible by $p^s$ and, for $0\leq j\leq r-1$, 
\[
\binom{j}{j} d_{p,j}+
\binom{j+1}{j} d_{p,j+1}+
\dots 
\binom{p-1}{j} d_{p,p-1}\equiv 0 \pmod {p^{s+1}}.
\]
\end{enumerate}
\end{prop}

The last one is a specific constraint on $d_{2,j}$, asserting that under certain conditions the 
partition must be very unbalanced in terms of parity.

\begin{prop}
If $m=2^q-1$ and $2^q$ divides $n$, then $\left| d_{2,0}-d_{2,1}\right| \geq 2^{2^q-1}$.
\end{prop}

This result is proved in \cite{littlewood} using a result from \cite{lidl}. We will give a self-contained proof here. 
First, an easy lemma from~\cite{boyd2}.

\begin{lemma} If $f\in \calP (n,m)$ and $f(-1)\neq 0$, then $\left| f(-1)\right| \geq 2^m$.
\end{lemma}

\begin{proof} Write $f(x)=(x-1)^mg(x)$. 
Then $0<\left| f(-1)\right| =2^m\left| g(-1)\right|$, where $g(-1)\in \mathbb{Z}$.
Therefore $\left| f(-1)\right| \geq 2^m$. 
\end{proof}

The next proposition is due to Berend and Golan \cite{littlewood} and extends \cite[Cor.~1]{boyd2}.

\begin{prop} If $n=t u$, with $t$ a power of 2 and $u$ odd, and $f$ is an element 
of $\calP (n,m)$ that is divisible by $x+1$, then $m<t-1$.
\end{prop}

\begin{proof} We are given that $f(x)=(x-1)^m (x+1) h(x)$ for some polynomial $h$.
From the proof of Divisibility above, we know that $x^n-1\equiv  (x-1)^t g(x) \pmod 2 $, 
where $g(x)$ is not divisible by $x-1$ modulo~2.
The result follows immediately:
\[
x^n-1\equiv (x-1)f(x)=(x-1)^{m+1}(x+1)h(x)\equiv (x-1)^{m+2} h(x)\equiv (x-1)^t g(x)
\]
so that $m+2\leq t$, or $m<t-1$, as claimed. 
\end{proof}

Proposition 4.4 then follows from Propositions 4.5 and 4.6, the Equivalence Lemma, 
and the fact that $\left| d_{2,0}-d_{2,1}\right| =\left| f(-1)\right|$, 
where $f$ is the LP defined from the equipowerful set $A$.

\section{Nonexistence of Equipowerful Bipartitions}\label{Nonexistence}

Here we show how the constraints of $\S $4 yield negative results. 
Theorem 5.1, together with the positive results in $\S $6, will prove Theorems 3.1, 3.2, and 3.3.

\begin{thm} { }$168,184\notin L_7$; $176,200,216,232\notin L_8$; $208,224\notin L_9$; and $192\notin L_{10}$.
\end{thm}

\begin{proof} Because 16 does not divide 200, 216, or 232, these follow from Divisibility.
The cases of 208 and 224 are handled as follows, where $m=9$. 
Let $p=3$. Then $s=4$, $r=1$, and by Lemma 4.3 and Proposition 4.4, $\left| d_{3,1}\right| \leq 75$
and $d_{3,1}$ is divisible by 81. 
Therefore $d_{3,1}=0$, contradicting the fact that $C_{3,1}([n])$ is odd. 
For 192 and $m=10$, consider $p=5$; then $s=r=2$. 
By Lemma 4.3, $\TI{d}_5$ must be $(\pm 25,\mp 25,0,0,0)$. 
But then the $j=1$ case of Proposition 4.3(2) looks at the sum of $(\mp 25,0,0,0)$, 
which is not divisible by $p^{s+1}$, or 125.

The case of 176 takes more work. Consider first $p=2$, 3, and 5. For these cases, $r=0$. 
A search using the constraints in $\S $4 shows that $\TI{d}_2=(0,0)$, 
$\TI{d}_3=(\pm 27,\pm 27,\mp 54)$ (replacing $A$ with its complement if
necessary, we may assume $d_{3,0}=27$), and $\TI{d}_5$ is one of 16 vectors. 
These cases lead to residue counts as follows, where the first entry is for residue 0: 
$p=2$: $(44,44)$; $p=3$: $(43,43,2)$, and $p=5$:
\begin{align*}
& (23,10,10,10,35),(23,10,10,35,10),(23,10,35,10,10),(23,35,10,10,10),(18,5,5,30,30),\\
&(18,5,30,5,30),(18,5,30,30,5),(18,30,5,5,30),(18,30,5,30,5),(18,30,30,5,5),(28,15,15,15,15),\\
&(13,0,25,25,25),(13,25,0,25,25),(13,25,25,0,25),(13,25,25,25,0),(8,20,20,20,20).
\end{align*}
All the $p=5$ cases, except the last, resolve as follows. Consider $(28,15,15,15,15)$. 
There are 36 numbers in $[176]$ that are $0 \bmod 5$ and 24 of those are congruent to 0 or $1 \bmod 3$. 
But we need 28, so there must be at least four that are congruent to $2 \bmod 3$, violating
the 2 in $(43,43,2)$. 
This type of counting argument settles the first 11 cases. The next four cases are similar, 
where one contradicts the 2 by considering two classes with 25 elements each. 
That leaves only the case of $(8,20,20,20,20)$. To finish we look at the constraints for $p=7$ and 11. 
We can filter them down as was done for $p=5$, leaving nine choices for $\TI{d}_7$ 
and 82 for $\TI{d}_{11}$. 
When ILP is set to work on all $9\cdot 82=738$ possibilities with these five primes, and with 
power identities up to exponent 6, none of the cases leads to a solution. So $176\notin L_8$.

The remaining two cases, 168 and 184, are more complicated, but yield to a detailed 
computer-aided analysis of the constraints for small primes. We start with $184$. 
Let $n=184$ and $k=6$ and assume that $A$, $B$ witness $L_7(184)$; so $\left| A\right| =\left| B\right| =92$. 
We will consider the primes $p\leq 13$, learning all possibilities for $\TI{d}_p$ 
in each case, taking into account previous cases as we move up. Then at $p=17$ there will be no possibility 
for $d_{17,0}$ consistent with the results for smaller primes, proving $184\notin L_7$.

We start with $p=2$; then $s=7$ and $r=0$. We have
\begin{itemize}
\item  { }$\left| d_{2,0}\right| \leq 92$ by Lemma 4.2;
\item  { }$d_{2,1}=-d_{2,0}$, by Lemma 4.2;
\item  { }128 divides $d_{2,1}-d_{2,0}$, by Proposition 4.3;
\item  { }$\left| d_{2,1}-d_{2,0}\right| \geq 128$, by Proposition 4.4.
\end{itemize}

The first three mean that $\TI{d}_2$ is $(64,-64)$, $(0,0)$, 
or $(-64,64)$. The last item eliminates $(0,0)$.
Switching $A$ and $B$ if necessary, we can assume that $d_{2,0}\geq 0$. 
This proves $\TI{d}_2=(64,-64)$, which means
that the even-odd distribution in the witnessing sets is $(78,14)$ for $A$, 
and $(14,78)$ for $B$. We can invoke this switching trick once only.

Now let $p=3$; then $s=3$ and $r=1$. We have

\begin{itemize}
\item { }$\left| d_{3,j}\right| \leq 62$, by Lemma 4.2;

\item { }$d_{3,0}+d_{3,1}+d_{3,2}=0$, by Lemma 4.2;

\item { }$d_{3,0}$ is even and $d_{3,1}$, $d_{3,2}$ are odd by Lemma 4.2, 
because $(C_{3,0},C_{3,1},C_{3,2})=(62,61,61)$;

\item { }$27$ divides each $d_{3,j}$, by Proposition 4.3
\end{itemize}

These conditions mean that $d_{3,2}=\pm 27$; then only these vectors satisfy all four conditions: 
$(-54,27,27)$, $(54,-27,-27)$, $(0,-27,27)$, and $(0,27,-27)$. 
But the ones involving $\pm 54$ fail when the constraint for $p=2$ is considered, as follows. 
When $d_{3,0}=-54$, $B$ must have exactly 58 multiples of 3. 
But there are only 31 odd numbers in $[n]$ divisible by 3 and $B$ has only 14 evens by the $p=2$ work.
So if all 31 are in the odd part of $B$, and all of the 14 evens in $B$ are divisible by 3, 
the total is $31+14=45$, short of the needed 58.
A similar argument, interchanging $A$ and $B$, and even and odd, eliminates $d_{3,0}=54$, 
and so $\TI{d}_3$ must be one of $(0,\pm 27,\mp 27)$.

For larger primes, we use ILP; all the constraints are easily programmable. 
One starts by finding the feasible values of $d_{p,p-1}$.
For each one of those one finds the feasible values of $d_{p,p-2}$. 
Once we have $(d_{p,r},d_{p,r+1,}\dots ,d_{p,p-1})$, we can use
Proposition 4.3 (and the other constraints) to quickly find all feasible extensions 
to the full vector $\TI{d}_p$. 
So we need only work down to $d_{p,r}$. When $p=5$, this yields that 
$\TI{d}_5$ is either $(15,5,-5,-15,0)$ or $(5,-15,15,-5,0)$.
Moreover, this case eliminates one of the $\TI{d}_3$ vectors, 
leaving $\TI{d}_3=(0,-27,27)$. 
The next case gives $\TI{d}_7=(7,-7,0,0,0,0,0)$. 
And then $p=11$ gives $\TI{d}_{11}=(-3,-1,3,-5,5,-3,1,3,0,0,0)$
(and this eliminates the second $\TI{d}_5$ vector). 
We next get $\TI{d}_{13}=(1,-1,2,-2,0,2,0,0,0,-2,0,2,-2)$,
and when we move to $p=17$, we find that there are no feasible values of $d_{17,16}$.

And now the last case: suppose $A,B$ is a 7-equipowerful bipartition of $[168]$. 
As was done for 184, we can deduce that $\TI{d}_2=(-64,64)$, 
$\TI{d}_3=(0,0,0)$, 
$\TI{d}_5=(\mp 10,0,\pm 10,\mp  5,\pm 5)$, and 
$\TI{d}_7=(0,0,0,0,0,0,0)$ (there are $42+119=161$ other possibilities for 
$\TI{d}_7$ but they are proved infeasible when
the constraints for $p=11$ are brought into play). 
When $p=11$, we have $r=\text{mod}(m,10)=7$, and this means we need consider only vectors
of length $p-r=4$; they extend to 11-vectors using Proposition 4.3. 
There are 65536 possible quadruples and they extend to a set of 301388 11-vectors.
Now, because $\TI{d}_2=(-64,64)$, we know that $A$ has 
10 evens and 74 odds (and vice versa for $B$). 
Let $\TI{D}$ be one of the 301388 possibilities for 
$\TI{d}_{11}$. 
Let $A_\TI{D}$ be the counts in the residue classes mod 11 in $A$ determined by $\TI{D}$. 
If some $a\in A_\TI{D}$ has the form $8+q$ ($q\geq 1$), 
then because at most eight entries in the 11-residue class of $A$ can lie in 
the odds of $A$, $q$ of them must lie in the evens in $A$. 
If the sum of the $q$-values over such entries $a$ exceeds 10, the even count of $A$, 
we know that $\TI{D}$ is infeasible. 
The same argument applies to $B$, with odds instead of evens. 
Further, for any pair of residue classes mod 11, $[168]$ has at most 16 numbers 
among the odds congruent to one of the two residues mod 11. 
So if $\TI{D}$, as above, forces two residue classes to have 
$10+16+1=27$ or more elements in $A$, we have a contradiction. 
And the same applies to $B$. Filtering the 301388 possibilities for 
$\TI{D}$ leads to only 2640 vectors. 
It takes about 22 seconds for ILP to check each one against the constraints for
$p=2,3,5$, and $7$ (using the first choice for $\TI{d}_5$); 
of the 2640 vectors, 309 turn out to be feasible (this takes
about 18 hours), so we then move to ILP with sum constraints added. 
Using the power identity up to exponent 4 is usually enough, but sometimes (34 cases) 
all the powers (up to 6) were needed. All turn out to be infeasible. 
Then the 2640 vectors are put through the same grinding machine with the other choice 
for $\TI{d}_5$, and the results are the same, proving $168\notin L_7$. 
(More detail in the second case: the sieving of vectors works thus: $301388\to 2640\to 245\to 23$.) 
The complete proof for $n=168$ took about two days of computation using \textit{Mathematica}{'}s 
ILP function, which calls COIN-OR. 
\end{proof}

Positive results regarding $L_m$ (as in the next section) can be certified correct by 
simple arithmetic in an instant. But we have no idea of how to succinctly certify negative results.
The 168 result had been proved earlier by the second author, relying on ILP as in 
the 184 case and requiring a few days on a cluster of 100 computers. 
It is important that this case gave the same results when carried out on two entirely 
different platforms and using somewhat different algorithms. 
For ILP work, the second author uses lp$\_$solve, the third author uses 
the \textit{SAS} MILP solver, and the fourth author uses \textit{Mathematica}.

\section{Existence of Equipowerful Bipartitions}\label{Existence}

Here we will find equipowerful sets that prove Theorems 3.1, 3.2, and 3.3. 
Symmetry plays a key role even in Theorem 3.1, where it is not explicitly mentioned. 
All of our computational evidence supports the following idea.

\begin{symconj} If there is an LP of length $n$ and exact order~$m$, then there is one that
is $(-1)^m\,$symmetric.
Alternatively, if there is an equipowerful bipartion of length $n$ and exact order~$m$
then there is a witnessing set that is $(-1)^m\,$symmetric. 
\end{symconj}

For instance, the 3-equipowerful set $\{0,2,6,7,8,10\}\subset [12]$ given in $\S $4 is antisymmetric. 

This idea has two important implications::
(1) Searches should be streamlined (i.e., made feasible) by just assuming $(-1)^m\,$symmetry, as needed.  
(2) The Symmetry Conjecture should be tested in all situations where we know that
$\calP (n,m)$ is nonempty, and testing is feasible.

In fact, many of our searches would not have been possible without making the assumption in (1)
Moreover, we were able to succeed for all of the (infinitely many) $n$ implicitly asserted in the 
tables of $L_m$ above by using the regenerative pairs to be described shortly.

The characterizations for $m\leq 6$ (Table 1) were all known, but the symmetry 
aspect of Theorem 3.2 is new for $3\leq m\leq 8$. 
To prove Theorem 3.1 (as well as the known characterizations for $m\leq 6$), 
one can start with the trivial $m=0$ case and use Doubling in Theorem~\ref{basic} to move up, 
while finding enough additional examples so that Addition in Theorem 4.1 
leads to the infinite family of witnessing sets. 
For example, when $m=2$, the base case is $\{0,3\}\subset [4]$ and Addition handles the rest. 
When $m=3$, doubling the $m=2$ case gives $8+8\N$, leaving $12+8\N$ unresolved. 
But once we have the $n=12$ witness, Addition ($12+8$, $12+8+8$, $\dots$) takes care of the rest
so that this method handles all $m\leq 8$. 
Since Addition does not preserve symmetry, this method will not 
yield sets having the desired symmetry properties.

\begin{table}[ht]\label{sporadic}
$
\begin{array}{|l|l|}\hline
 m &  \\ \hline
 3 & 12 \\ \hline
 4 & 24 \\ \hline
 5 & 40, 56 \\ \hline
 6 & 48, 72, 88, 104 \\ \hline
 7 & 112, 200, 216, 232, 248, 264, 280, 296, 312, 328 \\ \hline
 8 & 144, 208, 240, 272, 304 \\ \hline
 9 & 192, 240 \\ \hline
\end{array}
$
\vspace*{0.1in}
\caption{Sporadic cases for symmetry.}
\end{table}

We found sets for all the needed sporadic cases, which are listed in Table~3.
The largest example shows $328\in L_7$ and is the antisymmetrization of 
\begin{align*}
\{& 0,3,5,7,9,11,12,15,16,18,19,21,23,27,29,30,33,35,37,39,41,42,45,47,\\
  & 49,51,53,55,57,59,62,63,65,66,67,69,71,73,75,77,81,83,84,87,89,91,93,95,\\
  & 97,99,100,102,106,107,111,112,113,114,116,117,119,122,125,128,129,130,131,\\
  & 133,134,137,139,143,144,146,149,150,152,153,160,161,162,163\}.
\end{align*}
The sum of the 6th powers of this set is $28863168757954570$, which equals 
the 6th power sum for its complement in $[328]$; 
this power-sum equality holds for all powers up to~$6$. 
The corresponding LP factors into the product of $(1-x)^7$ and an irreducible 
polynomial of degree 320. 
The witnesses for all 25 needed cases follow from the data in Tables 4 and 5.

The raw search space for the 368 example has more than $10^{97}$ sets, so clearly 
some tricks are needed to get the search to work. 
The main tool is integer linear programming: a binary variable is used for each value 
in $[n]$ and the power-sum constraints are then linear equations.
In addition we use the following constraints and tricks.

\noindent 1. { }We assume the Symmetry Conjecture, which halves the variable count. 
By Symmetry in Theorem~\ref{basic}, this means that the power constraints need
only go to the ($m-2$)nd power, as the last one comes for free.

\noindent 2. { }We use the modular constraints for small primes derived from the results in $\S $4. 
Further, we filter the constraints down to ones that are consistent with the 
assumed symmetry property.
\pmb{  }Recall that the constraints are first derived for the vectors $\bar{d}_p$; they
are then used to get the counts for the congruence classes in the set $A$. 
To filter the set of $\TI{d}_p$ as needed for $(-1)^m$ symmetry, 
keep only those $\TI{d}_p$ for which $d_{i,p}=(-1)^md_{n-i-1,p}$ 
where $0\leq i\leq p-1$ and the indices are reduced modulo~$p$.

\noindent 3. { }To avoid problems with the very large numbers that arise, we shift the domain 
from $[n]$ to the interval $\left[-\frac{n}{2}+1,\frac{n}{2}\right]$.
This is allowed because if $A\overset{m}{=}B$ as subsets of $[n]$, then 
$A+t\overset{m}{=}B+t$, a fact that is easily proved by using the Equivalence
Lemma and observing that multiplication by $x^t$ does not affect the power 
of $1-x$ that divides a polynomial.

\noindent 4. { }For the same reason as in (3), we replace powers by binomial coefficients. 
The definition of an $m$-equipowerful number uses powers $x^j$, but any family 
of $m$ polynomials of degree $0,1,\dots ,m-1$ that takes integer values 
on integer arguments can be used instead. 
If the ILP search uses as its main constraint not power identities but instead 
equalities over $A,B$ of the binomial coefficient polynomials
then the size of the numbers is substantially reduced. 
Consider the search to show $240\in L_9$; using powers (and also points 1 and
3 above) involves numbers as large as $120^8$, about $4\cdot 10^{16}$, compared to the
binomial coefficient $\binom{120}{8}$,
which is about $8\cdot 10^{11}$.

Addition from Theorem~\ref{basic} fails to preserve symmetry, so a proof of Theorem 3.2 
requires a new type of rule that does. 
This is what the concept of \textit{regenerative pairs} accomplishes. 
We stumbled on this idea when we realized that the Symmetry Conjecture had not
been proved even for the case $m=3$. 
To handle that case we needed a new way to go from 12 to $20,28,\dots$. 
We found a way to do this and then found several other instances where 
$f\in \calP (n,m)$ could be extended to an LP in 
$\calP (n+n',m)$ so that symmetry is preserved. 
In short, we found a new type of ``addition rule'' that respects symmetry.

\begin{thm}\label{RPthm} 
Assume that $f$ and $g$ are LPs of lengths $n$ and $\delta$.  Fix a positive integer~$m$
and let $s = (-1)^m$ and $S$ be the corresponding $s$-symmetrization operator.
Define LPs
\[
f_1 = f \lor g, \qquad f_2 = f \lor g \lor sg^*
\]
of lengths $n+\delta$ and $n+2\, \delta$.
Then 
\[
S(f) \in \calP (2 n,m) \; \text{and} \; S(f_1) \in \calP (2n + 2\delta,m) 
\quad \text{ imply that } \quad S(f_2) \in \calP(2n+4\delta,m).
\]
\end{thm}

In words: if an LP $f_1$ extends an LP $f$, and both $S(f)$ and $S(f_1)$ have order $m$, then
there is a closely related $f_2$, extending $f_1$, such that $S(f_2)$ has order~$m$.
Continuing in this way gives a sequence of
LPs of order $m$ whose lengths form an arithmetic progression.

\begin{defn}\label{RPdef}
If $f$, of length~$n$, and $f_1 = f \lor g$, of length $n_1 = n+\delta$, satisfy the hypotheses of the 
theorem, then $(f,f_1)$ is
said to be a Regenerative Pair (RP) for $(2n,2n_1)$.
\end{defn}

\noindent \textit{Examples.}  1. The LPs $1-x+x^2-x^3-x^4-x^5$ and $1-x+x^2-x^3-x^4-x^5+x^6+x^7+x^8+x^9$ 
are an RP for $(12,20)$: the antisymmetrizations of the polynomials are in 
$\calP (12,3)$, $\calP (20,3)$, respectively.

\noindent 2. If $f$ is such that $S(f)\in \calP (2\, n,m)$, where $S$ is $(-1)^m\,$symmetrization, 
then $(f,F)$ is an RP for $(2\, n,4\, n)$, where $F$ is the 
length-$2n$ initial part of $S(f)\lor -S(f)$.

\noindent 3. For all $m$, the pair $\tau_m,\tau_{m+1}$ is an RP for $(2^{m+1},2^{m+2})$.

\begin{proof}  Throughout this proof $\equiv $ denotes congruence modulo~$(x-1)^m$.
Let $n_1 = n + \delta$ and $n_2 = n+ 2\,\delta$ be the lengths of $f_1$ and $f_2$.
From the hypotheses,
\begin{align*}
S(f)  & = f+s x^n f^* \equiv 0 \\
S(f_1) = f_1+s x^{n+\delta}f_1^* & = f + x^n g + s\,x^{n+\delta} g^* + s\, x^{n+2\delta} f^* \equiv 0. 
\end{align*}
Subtracting the first congruence for $S(f)$ from the one for $S(f_1)$,
and dividing by $x^n$, gives 
\[
-s\, f^*+g+s\, x^{\delta }g^*+s\, x^{2\delta }f^*\equiv 0.
\] 
Moving the $f^*$ terms to the right side gives 
$S(g)\equiv s\, (1-x^{2\, \delta })f^*$.
Multiplying by $1+x^{2\delta }$ then gives 
\[
\left(1+x^{2\, \delta }\right)S(g)\equiv s\, (1-\, x^{4\, \delta })f^*.
\]
Multiply by $x^n$, replace $s\, x^n\, f^*$ by $-f$, bring everything to the left
side, and expand to get
\[
f+x^ng+s x^{n+\delta }g^*+x^{n+2\delta }g+s\, x^{n+3\delta }g^*+s\, x^{n+4\, \delta }f^* \equiv 0.
\]
This can be carefully checked, from the definition of $f_2$, to give the 
desired conclusion: $S(f_2)~\equiv~0$.
\end{proof}

The power of this theorem is that $(f_1,f_2)$ becomes an RP, with lengths $n_1$ and $n_1+\delta$, so
the theorem can be applied again.
We can iterate forever, concluding that 
$L_m$ contains $2n,2n+2\delta ,2n+4\delta ,2n+6\delta ,\dots $. 
Moreover, all the witnesses will be $(-1)^m\,$symmetric.

\begin{cor} With notation as in Theorem \ref{RPthm}, for any $j\geq 0$, define
\[
G_j = \bigvee_{i=1}^j\,h_i,
\]
where $h_i=g$ if $j$ is odd and $h_i=s\, g^*$ if $j$ is even. 
Then $S(f \lor G_j)$ is in $\calP(2b+2j\delta,m)$ and is $(-1)^m\,$symmetric.
\end{cor}

Continuing with the $(12,20)$ example, the antisymmetrization of $f$ corresponds to 
$\{0,2,6,7,8,10\}$, while the same for $F$ gives $\{0,2,6,7,8,9,14,15,16,18\}$. 
These two sets are antisymmetric witnesses to $12,20\in L_3$. 
Because the left halves of these sets are nested, the iterative construction of 
Theorem 5.1 leads to the single infinite set $X=\{0,2,6,7,8,9,14,15,16,17,22,23,24,25,\dots \}$. 
This single set provides antisymmetric witnesses for $12+8\N$: just take the 
appropriate initial segment and antisymmetrize it. 
The difference sequence of $X$ (assuming 0 is in the set) is almost periodic: 
$2,4,\overline{1,1,1,5}$, where the bar indicates repetition. 
This finitary method of witnessing infinitely many numbers in $L_m$ will occur 
whenever we have an RP: there will be a single almost periodic difference sequence, which defines an
infinite set $X$ that is a union of finitely many arithmetic progressions. 
Using $(-1)^m$ symmetrization on appropriate initial segments of $X$
will yield $(-1)^m$ symmetric witnesses for infinitely many values in $L_m$.

So we can prove Theorems 3.2 and 3.3 by finding RPs for the needed cases. 
It took several weeks, but the ILP method with the various constraints succeeded in 
finding all the required pairs. 
For each $m\leq 6$ the asymptotic result follows from a single RP (Table 4). 
But it took eight RPs to cover $L_7$. 
Doubling the $m=7$ case covers almost all of $L_8$, and the characterization 
of that case is completed by finding seven additional 8-equipowerful sets, shown in Table 5.

In order to give all of the data needed for our theorems in a small amount of space,
we will encode the polynomials in hexadecimal as follows.
Convert the characteristic function of $A\subseteq [n]$ into a binary string,
padded on the right with 0s so that the bit count is a multiple of 4, and then convert to hex.

For example, the RP for $(12,20)$ is defined by $\{0,2,6,7,8,9,14,15,16,18\}\subset [20]$. 
We need only consider the left half $X=\{0,2,6,7,8,9\}$, as the full set is the 
antisymmetrization of $X$. 
The corresponding bit-string, padded to 12 bits, is 
$1010\hspace{0.05in} 0011\hspace{0.05in} 1100$. 
The hex version of this is $\text{A3C}$.

Table 4 shows all the needed RPs for the asymptotic results (the $m=8$ case follows 
by just doubling the sets from $m=7$); the first four cases include the sets. 
Table 5 shows individual examples for the cases not covered by the pairs. 
In all cases only the left half of the sets is encoded, as the full set is obtainable by $(-1)^m$ symmetrization.

\begin{table}[ht]\label{bigRPlist}
$
\begin{array}{|l|l|l|} \hline
m  & \text{lengths}      & \text{hex code for regenerative pair} \\ \hline
3  & 12 , 20   & \text{A3C}  \hfill                                \{0,2,6,7,8,9\} \\ \hline
\{
4  & 24 , 32   & \text{995C}  \hfill                                   \{0,3,4,7,9,11,12,13\} \\ \hline
5  & 48 , 64   & \text{A4DD233C}  \hfill               \{0,2,5,8,9,11,12,13,15,18,22,23,26,27,28,29\} \\ \hline
5  & 56 , 72   & \text{96A371999}  \hfill          \{0\,3\,5,6,8,10,14,15,17,18,19,23,24,27,28,31,32,35\} \\ \hline
6  & 104 , 120 & \text{C32F8696687E155} \\ \hline
6  & 112 , 128 & \text{936D342C73B58A0F} \\ \hline
7  & 208 , 224 & \text{A55A9936CCA5363A65E41B6CC35A} \\ \hline
7  & 200 , 328 & \text{9559B51655655553755459555A31ED24F651A6C0F} \\ \hline
7  & 216 , 344 & \text{953D547143575D547554D551077417ED3C9A6982C37} \\ \hline
7  & 232 , 360 & \text{95555D95510D77154D75994531D714F238C7EB8831D1D} \\ \hline
7  & 248 , 376 & \text{C557185D75534395571D14715DC7534EC412F9C933CB199} \\ \hline
7  & 264 , 392 & \text{D29535559352D575954B549555A54B555C968936B1BDC6538} \\ \hline
7  & 280 , 408 & \text{A8ADA3CAE2B88E2B9AC32CAE2B8EE2B8AC347D930FE4826A3C7} \\ \hline
7  & 296 , 424 & \text{9996356B5496555555555E925945DA7532555477C255A9AA6587A} \\ \hline
7  & 312 , 440 & \text{A56A69555555B59550965AD5B71515554E56725A64EA946FB10D563} \\ \hline
\end{array}
$
\vspace*{0.1in}
\caption{RPs that suffice, asymptotically, to give all witnessing sets for $L_m, m \le 8$.}
\end{table}

\begin{table}[ht]\label{sporadics}
$
\begin{array}{|ll|l|}\hline
m  &   n    & \text{hex code for witness to} \quad n \in L_m                                                 \\ \hline
5  &  40    & \text{C1EC9} \hfill \{0,1,7,8,9,10,12,13,16,19\}                                                 \\ \hline
6  &  48    & \text{C27D8C}  \hfill \{0,1,6,9,10,11,12,13,15,16,20,21\}                                        \\ \hline
6  &  72    & \text{998EAA3C5} \quad \hfill \{0,3,4,7,8,12,13,14,16,18,20,22,26,27,28,29,33,35\}                     \\ \hline
6  &  88    & \text{96C362DE4A3} \; \hfill                                                                 \\ \hline
8  &  144   & \text{A962D5AF05357231CD}                                                                      \\ \hline
8  &  208   & \text{C5A2CB6B945B28BC43ADB6586C}                                                              \\ \hline
8  &  240   & \text{8DE3424EBD39684AD636AD8932C4F9}                                                          \\ \hline
8  &  272   & \text{D0EC315B66B4DE6110D339D63B4CD61BC4}                                                      \\ \hline
8  &  304   & \text{A9E21ACCB5C96794C56D9C4255B4336DEAD8A4}                                                  \\ \hline
8  &  336   & \text{92D6F430CB2AD476994F60DEB609EA1158973CA7D2}                                              \\ \hline
8  &  368   & \text{B4D121ABDED920EAF02CE76086B55BA6995ABE9A43247A}                                          \\ \hline
9  &  192   & \text{C1BE1E21CD63D295A7887A59}                                                                \\ \hline
9  &  240   & \text{9666995C93C3AA5935E81CB7C2938D}                                                         \\ \hline
\end{array}
$
\vspace*{0.1in}
\caption{Sporadic examples.}
\end{table}

The method of proof using RPs leads to a surprising amount of structure in $L_m$ when $m\leq 7$. 
Consider the trivial fact that, for $L_2$, the symmetry result can be proved by a single set. 
Just let $X=\N$ and get the witness for $4\, k$ by taking the first $\,
k$ entries in $X$ and symmetrizing the result. 
For $L_3$, we have two RPs that yield antisymmetric witnesses for all cases: 
the 12/20 case and the 8/16 case.  For $L_4$ a single RP takes care of $24+8\N$ 
and that leaves only the singleton 16, which, by the example following Definition~\ref{RPdef},
can be viewed as the first half of an RP. If we ignore the small number of exceptions, 
we see from Table 4 that up to $m=6$ we have at most two RPs that cover $L_m$. 
At $m=7$ we need nine: eight to cover the numbers that do not arise by doubling $L_6$, 
and one more to cover those that do arise by such doubling. 
However, for $m=8$ we were unable to find any RPs after trying several cases for the first few values. 
So we can ask whether this covering set of finitely many RPs always exists.
As pointed out after Corollary 5.2, this structure means that there are finitely many sets $X_i$, 
each of which is a union of arithmetic progressions, so that the appropriate symmetrization of 
initial segments of the $X_i$ lead to $(-1)^m$ symmetric witnesses for all $n\in L_m$, with finitely 
many exceptions. 
Though, by Example 2 after Definition 6.2, any LP is the beginning of a trivial RP, so that
perhaps these aren't truly exceptions.

\section{Conclusion and Questions}\label{questions}

It is easy to make conjectures based on the patterns observed in data. 
This area is remarkable for the number of such guesses that have turned out to be wrong. 

For instance, looking at $L_m$ for $m\leq 5$ suggests the natural idea that the Thue--Morse polynomial 
$\tau_m$ will be the order-$m$ LP of smallest length, so that $2^m=\min  L_m$. 
This was disproved by Skachek \cite{skachek} and Boyd \cite{boyd1}, who found $48\in L_6$. 
Another conjecture arising from $m\leq 5$ is that $\tau_m$ is the unique (up to sign) LP in 
$\calP (2^m,m)$; this fails because, again by Boyd, 
$\calP (64,6)$ has three LPs (up to sign). 
Increasing $m$, one sees that for $m\leq 9$, $\min  L_m>2^{m-1}$, and one might be tempted
to guess that this is always true.  Although nothing in our data immediately contradicts this, note
that if $f\in \calP (144,8)$, then by Multiplication in Theorem~\ref{basic},
$f \# f\in \calP (144^2,16)$ and $2^{15}>144^2\in L_{16}$.

At one point it seemed natural to ask whether $2^{m-1}$ was never in $L_m$.
Richard Stong found a clever way to combine known elements of various $\calP(n,m)$ to disprove this; here,
with his kind permission, is his result. 

\begin{thm}\label{stong}
$\calP (2^{51},52)\neq \varnothing$.
\end{thm}

\begin{proof}
Let $p_{n,m}$ be any element of $\calP (n,m)$ and let $p^{\#j}$ 
denote $p\# p\#\, \cdot \cdot \cdot \, \#\, p$, with $j$ terms. 
Recalling that $p_{n,m}\#p_{n',m'}\in \calP (n n',m+m')$, we can get four large LPs (in fact, gigantic
compared to anything discussed earlier) as follows:
\begin{align*}
 a & =p_{144,8}^{\text{$\#$2}}\#p_{192,9}^{\text{$\#$4}}=p_{28179280429056, 52}\\
 b & =p_{48,6}^{\text{$\#$3}}\# p_{112,7}\# p_{192,9}^{\text{$\#$3}}=p_{87668872445952, 52} \\
 c & =p_{16,4}\# p_{112,7}^{\text{$\#$2}}\# p_{208,8} \# p_{272,8}\#p_{192,9}^{\text{$\#$2}}=p_{418591807635456,
52} \\
 d & =p_{112,7}\# p_{208,8}^{\text{$\#$4}}\# p_{8192,13}=p_{1717359853174784, 52} \\
\end{align*}
Since
\[
 28179280429056+87668872445952+418591807635456+1717359853174784=2^{51},
\]
it follows that $a\lor b\lor c\lor d=p_{2^{51},52}$, as desired.
\end{proof}

Note that all terms in the $2^{51}$ equation are divisible by $2^{32}$. 
This relation was found by compiling a useful list of rational numbers $n/2^m$ for which 
$\calP (n,m)\neq \varnothing$, and then doing a search (in \textit{Mathematica}) 
for a subset that summed to $1/2$. 

Allouche and Shallit in \cite[Open Problem 6.12.5]{shallit} raise the question of whether, roughly,
$\tau_m$ has the smallest ``error'' of all elements $f(x) = \sum a_ix^i$
of $\calP(2^{m-1},m)$.  This error is defined to be the (absolute value of the) $m$th moment
of the corresponding set bipartition, which is
\[
f^{(m)}(1) = \sum a_i \, i(i-1) \ldots (i-m+1) = \sum a_i \, i^m,
\]
where the last equality holds because all of the smaller moments (and corresponding derivatives) are 0.
(In particular, the $m$th moment is the first nonzero moment of $\tau_m$.)
From $\tau_m(x) = (1-x)(1-x^2)\dots (1-x^{2^{m-1}})$ one works out that
\[
\tau_m^{(m)}(1) = (-1)^m m! \, 2^{0+1+\ldots +(m-1)} = (-1)^m m! \, 2^{m(m-1)/2}.
\]
However, Theorem~7.1 gives a remarkably definitive answer to the Open Problem: the $p_{2^{51},52}$ constructed
in the theorem has $m$th moment equal to~$0$.

Given the sobering record of failed guesses and conjectures discussed above, caution is in order.
So we end with a list of questions, and are only willing to label the first as a conjecture.
\begin{itemize}
\item (Symmetry Conjecture) If $\calP (n,m)\neq \varnothing$, then it contains a $(-1)^m$ symmetric LP. 
(For $m\leq 8$, this holds by Theorem 3.2.) 
\item Is it the case that, for each $m$, there is a finite family of RPs that provides 
$(-1)^m$ symmetric witnesses for each entry in $L_m?$ 
(For $m\leq 7$, the answer is yes from the proof of Theorem 3.2.) 
\item If $n=\min  L_m$, are all $f\in \calP (n,m)$ $(-1)^m$, with maximal order $m$, $(-1)^m\,$symmetric? 
(This is true for $m\leq 6$.) 
\item What is $m^*(240)$? Is $240\in L_{10}$?   We can use ILP to show that there is 
no symmetric LP in $\calP (240,10)$, so the Symmetry Conjecture would imply that 
$m^*(240)=10$. 
Is $272\in L_9$? What is the smallest $k$ so that $k+16\N\subset L_9$?
\end{itemize}

\bibliography{main}

\begin{thebibliography}{10}

\bibitem{shallit}
J.-P. Allouche and J.~Shallit.
\newblock {\em Automatic {S}equences}.
\newblock Cambridge Univ.~Pr., New York, 2003.

\bibitem{littlewood}
D.~Berend and S.~Golan.
\newblock Littlewood polynomials with high-order zeros.
\newblock {\em Math. Comp.}, 75:1541--1552, 2006.

\bibitem{prouhet}
P.~Borwein and C.~Ingalls.
\newblock The {Prouhet}--{Tarry}--{Escott} problem revisited.
\newblock {\em Ens. Math.}, 40:3--27, 1994.

\bibitem{boyd1}
D.~W. Boyd.
\newblock On a problem of {Byrnes} concerning polynomials with restricted
  coefficients.
\newblock {\em Math. Comp.}, 66:1697--1703, 1997.

\bibitem{boyd2}
D.~W. Boyd.
\newblock On a problem of {Byrnes} concerning polynomials with restricted
  coefficients, \textrm{II}.
\newblock {\em Math. Comp.}, 71:1205--1217, 2002.

\bibitem{tarryescott}
H.~L. Dorwart and O.~E. Brown.
\newblock The {Tarry-Escott} problem revisited.
\newblock {\em Amer. Math. Monthly}, 44:613--626, 1937.

\bibitem{asymptotically}
G.~Freiman and S.~Litsyn.
\newblock Asymptotically exact bounds on the size of high-order spectral-null
  codes.
\newblock {\em IEEE Trans. Info. Th.}, 45:1798--1807, 1999.

\bibitem{golan}
S.~Golan.
\newblock Equal-moments division of a set.
\newblock {\em Math. Comp.}, 77:1695--1712, 2008.

\bibitem{lidl}
R.~Lidl and H.~Niederreiter.
\newblock {\em Finite {F}ields}.
\newblock Cambridge Univ.~Pr., New York, 1994.

\bibitem{skachek}
V.~Skachek.
\newblock {\em Coding for {S}pectral-null {C}onstraints}.
\newblock M.~Sc.~thesis (in Hebrew), Technion, 1997.

\end{thebibliography}
\bibliographystyle{plain}

\end{document}